\documentclass[11pt,english]{amsart}
\usepackage[T1]{fontenc}
\usepackage[latin1]{inputenc}
\usepackage{color}
\usepackage{amsthm}
\usepackage{amstext}
\usepackage{amssymb}

\makeatletter

\providecommand{\tabularnewline}{\\}

\numberwithin{equation}{section}
\numberwithin{figure}{section}
\theoremstyle{plain}
\newtheorem{thm}{\protect\theoremname}
  \theoremstyle{definition}
  \newtheorem{defn}[thm]{\protect\definitionname}
  \theoremstyle{remark}
  \newtheorem{rem}[thm]{\protect\remarkname}
  \theoremstyle{plain}
  \newtheorem{prop}[thm]{\protect\propositionname}
  \theoremstyle{plain}
  \newtheorem{cor}[thm]{\protect\corollaryname}
  \theoremstyle{plain}
  \newtheorem{conjecture}[thm]{\protect\conjecturename}
  \theoremstyle{plain}
  \newtheorem{lem}[thm]{\protect\lemmaname}
  \theoremstyle{definition}
  \newtheorem{example}[thm]{\protect\examplename}

\@ifundefined{definecolor}
 {\usepackage{color}}{}
\@ifundefined{definecolor}
 {\@ifundefined{definecolor}
 {\usepackage{color}}{}
}{}
\@ifundefined{definecolor}
 {\@ifundefined{definecolor}
 {\@ifundefined{definecolor}
 {\usepackage{color}}{}
}{}
}{}
\makeatother

\makeatother

\usepackage{babel}\addto\extrasfrench{\providecommand{\fg}{\ifdim\lastskip>\z@\unskip\fi~\frqq}}

\makeatother

\usepackage{babel}

\makeatother

\usepackage{babel}
  \providecommand{\conjecturename}{Conjecture}
  \providecommand{\corollaryname}{Corollary}
  \providecommand{\definitionname}{Definition}
  \providecommand{\examplename}{Example}
  \providecommand{\lemmaname}{Lemma}
  \providecommand{\propositionname}{Proposition}
  \providecommand{\remarkname}{Remark}
\providecommand{\theoremname}{Theorem}

\begin{document}
\global\long\global\long\global\long\def\Alb{{\rm Alb}}
 \global\long\global\long\global\long\def\Jac{{\rm Jac}}
 \global\long\global\long\global\long\def\Hom{{\rm Hom}}
 \global\long\global\long\global\long\def\End{{\rm End}}
 \global\long\global\long\global\long\def\aut{{\rm Aut}}
 \global\long\global\long\global\long\def\NS{{\rm NS}}
 \global\long\global\long\global\long\def\SSm{{\rm S}}
 \global\long\global\long\global\long\def\psl{{\rm PSL}}
 \global\long\global\long\global\long\def\CC{\mathbb{C}}
 \global\long\global\long\global\long\def\BB{\mathbb{B}}
 \global\long\global\long\global\long\def\PP{\mathbb{P}}
 \global\long\global\long\global\long\def\QQ{\mathbb{Q}}
 \global\long\global\long\global\long\def\RR{\mathbb{R}}
 \global\long\global\long\global\long\def\FF{\mathbb{F}}
 \global\long\global\long\global\long\def\DD{\mathbb{D}}
 \global\long\global\long\global\long\def\NN{\mathbb{N}}
 \global\long\global\long\global\long\def\ZZ{\mathbb{Z}}
 \global\long\global\long\global\long\def\HH{\mathbb{H}}
 \global\long\global\long\global\long\def\gal{{\rm Gal}}
 \global\long\global\long\global\long\def\OO{\mathcal{O}}
 \global\long\global\long\global\long\def\pP{\mathfrak{p}}
 \global\long\global\long\global\long\def\pPP{\mathfrak{P}}
 \global\long\global\long\global\long\def\qQ{\mathfrak{q}}
 \global\long\global\long\global\long\def\qQQ{\mathfrak{Q}}
 \global\long\global\long\global\long\def\mm{\mathcal{M}}
 \global\long\global\long\global\long\def\aaa{\mathfrak{a}}

\title[Involutions of second kind on Shimura surfaces]{Involutions of second kind on Shimura surfaces and surfaces of general type with $q=p_g=0$}

\author{Amir D\v{z}ambi\'{c}, Xavier Roulleau}

\maketitle

\begin{abstract}
Quaternionic Shimura surfaces are quotient of the bidisc by an irreducible
cocompact arithmetic group. In the present paper we are interested
in (smooth) quaternionic Shimura surfaces admitting an automorphism
with one dimensional fixed locus; such automorphisms are involutions.
We propose a new construction of surfaces of general type with $q=p_{g}=0$
as quotients of quaternionic Shimura surfaces by such involutions.
These quotients have finite fundamental group. 
\end{abstract}

\section{Introduction}

Among smooth minimal surfaces of general type, the ones with vanishing
geometric genus $p_{g}$ are of main interest (see e.g. \cite{bauercatanesepignatelli2}).
For such surfaces, the Chern number $c_{1}^{2}=K^{2}$ belongs to
the set $\{1,\dots,9\}$ and $c_{2}=12-c_{1}^{2}$. We are far away
from a complete classification, although great advances have been
done recently, e.g. for surfaces with $c_{1}^{2}=9$, the fake projective
planes, which have been completely classified, see \cite{Prasad}, \cite{Steger}.
In the other cases, a major task is to construct new examples of such
surfaces.

In this paper we give an uniform construction of surfaces with $q=p_{g}=0$
and $c_{1}^{2}=1,\dots,7$. These surfaces are obtained as quotients
of smooth quaternionic Shimura surfaces $X$ by a special kind of
involution. Recall that a smooth Shimura surface $X=X_{\Gamma}$ is
the quotient of $\HH\times\HH$, the product of two copies of the
complex upper half plane, by a discrete cocompact torsion free group
$\Gamma\subset\aut(\HH)\times\aut(\HH)$ of holomorphic automorphisms
defined by certain quaternion algebra. The invariants of $X$ are
$c_{1}^{2}(X)=2c_{2}(X)=8(1+p_{g}(X))$ and $q(X)=0$.

In the first step we prove that an automorphism of a Shimura surface
is quite special:
\begin{thm}
Let $\sigma$ be an automorphism of a smooth Shimura surface. Then
either $\sigma$ has a finite number of fixed points or the fixed
point set of $\sigma$ is a divisor and $\sigma^{2}=1$.
\end{thm}
Under certain conditions on the group $\Gamma$, the involution $\mu\in\aut(\HH\times\HH)$
exchanging the two factors induces an involution $\sigma$ on the
surface $X$. We call such an automorphism of $X$ an involution of
second kind. We obtain
\begin{thm}
Let $X$ be a smooth Shimura surface admitting an involution of second
kind $\sigma.$ The fixed point set $C$ of $\sigma$ is a a union
of disjoint smooth Shimura curves. The arithmetic genus $g$ of $C$
satisfies $2\leq g\leq p_{g}(X)$ and the quotient surface $Z=X/\sigma$
is smooth with finite fundamental group. If moreover $g=p_{g}(X)\leq8$,
$Z$ is a surface of general type with invariants:
\[
c_{1}(Z)^{2}=9-p_{g}(X),\, c_{2}(Z)=3+p_{g}(X),\, p_{g}(Z)=q(Z)=0.
\]

\end{thm}
If $p_{g}=2$ or $3$, one can prove that $g=p_{g}$ and therefore
$c_{1}^{2}=7$ and $6$ respectively. 

We then concentrate on the construction of such Shimura surfaces with
low $p_{g}$ and admitting an involution of second kind. It turns
out that such surfaces are rather exceptional. For instance, if we
restrict our consideration to totally real fields of degree 2, there
are at most 14 isomorphism classes of quaternion algebras leading
to smooth Shimura surfaces of geometric genus $2\leq p_{g}\leq8$
admitting an involution of second kind (see Theorem \ref{list_real_quadratic_candidates}).
On the other hand we consider Shimura surfaces corresponding to congruence
subgroups and we are able to show:
\begin{thm}
For $k=5,6$ there exists a smooth Shimura surface $X$ with an involution
of second kind $\sigma$ and $p_{g}(X)=k$.  In the case $k=5$,
the curve $C$ fixed by $\sigma$ is irreducible of genus $g(C)=5=p_{g}(X)$.
\end{thm}
In the light of some open questions concerning fundamental groups
of surfaces with geometric genus zero, see \cite{bauercatanesepignatelli2},
an example of a smooth Shimura surface with $p_{g}=2$ admitting an
involution of second kind would be highly interesting. However, finding
such a surface turns out to be very difficult (see Remark \ref{remark_deg_2}). 

This paper mixes two fields: Theory of Shimura surfaces and classical
algebraic geometry of surfaces. Shimura surfaces are closely related
to Hilbert modular surfaces (which were first systematically studied
by Hirzebruch, see for instance \cite{VanderGeer}, but are less known
and studied. It was also one of our aim to develop that theory of
Shimura surfaces.

The paper is organized as follows: In Section \ref{sec:Involutions and group gamma}
we discuss the conditions on $\Gamma$ under which the Shimura surface
$X_{\Gamma}$ has an involution of second kind. In Section \ref{sec:Quotient and fundamental group}
we study the quotient surface of a smooth Shimura surface by the action
of an involution of second kind and we prove that its fundamental
group is finite. In Section \ref{sec:Examples}, we investigate examples
of Shimura surfaces with low geometric genus admitting an involution
of second kind. In particular we develop new tools to create smooth
quaternionic Shimura surfaces and we make a systematic study of Shimura
surfaces defined over quadratic fields with low geometric genus. We
give examples of surfaces with $p_{g}=5,6$ admitting an involution
of second kind. Finally, in section \ref{sec:determination-of-the-curve}
we present the method to identify the Shimura curve fixed by an involution
of second kind acting on a Shimura surface and we furthermore examine
the example with $p_{g}=5$.

\section{Involutions of second kind acting on Shimura surfaces}

\label{sec:Involutions and group gamma}

\subsection{Quaternionic Shimura surfaces.\label{sub:Notations,-quaternionic-Shimura}}

Let us recall the construction of quaternionic Shimura surfaces.

Let $k$ be a totally real number field of degree $n=[k:\QQ]$. The
places of $k$ are the equivalence classes of valuations on $k$,
and the infinite places of $k$ correspond to embeddings $\sigma_{i}\in\Hom_{\QQ}(k,\RR)$,
$i=1,\ldots,n$.

Let $A$ be a division quaternion algebra whose center is $k$. For
every place $v$ of $k$, we denote by $k_{v}$ the completion of
$k$ with respect to $v$ and define $A_{v}=A\otimes_{k}k_{v}$. The
algebra $A$ is \emph{ramified} at $v$ if $A_{v}$ is a division
algebra over $k_{v}$ and \emph{unramified} otherwise, that is, if
$A_{v}\cong M_{2}(k_{v})$. By the classical theorem of Hasse and
the product formula for Hilbert symbols, the isomorphism class of
$A$ is uniquely determined by the set $Ram(A)$ of ramified places
of $A$. Assume that $A$ is unramified at the first, say $m\leq n$,
infinite places $\sigma_{1},\ldots,\sigma_{m}$ and ramified at the
remaining $n-m$ infinite places. Equivalently we assume that we have
an isomorphism 
\[
A\otimes_{\mathbb{Q}}\mathbb{R}\to M_{2}(\mathbb{R})^{m}\times\mathbf{H}^{n-m}
\]
 where $\mathbf{H}$ denotes the skew field of Hamiltonian quaternions.
Then, $A$ is uniquely determined up to an isomorphism by the tuple
\[
(k,\sigma_{1},\ldots,\sigma_{m},\mathfrak{p}_{1},\dots,\mathfrak{p}_{r}),
\]
 where $\mathfrak{p}_{1},\dots,\mathfrak{p}_{r}$ are the non-archimedean
places of $k$ where $A$ is ramified, and $\sigma_{1},\ldots,\sigma_{m}$
are the infinite places of $k$ such that $A\otimes_{\sigma_{i}(k)}\mathbb{R}\cong M_{2}(\mathbb{R})$.
Since we are interested in algebraic surfaces, we suppose that there
are exactly two infinite places $\sigma_{1},\sigma_{2}$ such that
$A$ is unramified at the $\sigma_{i}$. With this fixed ramification
behaviour at the infinite places, the isomorphism class of $A$ depends
only on the tuple $(\mathfrak{p}_{1},\dots,\mathfrak{p}_{r})$. Let
us write 
\[
A=A(k,\mathfrak{p}_{1},\dots,\mathfrak{p}_{r}),
\]
(omitting the explicit indication of two unramified infinite places)
for such a quaternion algebra in the following. The subgroup $A^{+}$
of $A$ consisting of the units of $A$ having totally positive reduced
norm can be identified via the isomorphism $A\otimes_{\QQ}\RR\stackrel{f}{\rightarrow}M_{2}(\mathbb{R})^{2}\times\mathbf{H}^{n-2}$
with a subgroup of $GL_{2}^{+}(\mathbb{R})^{2}\times\mathbf{H}^{*n-2}$,
and projecting to the first two factors gives an injection of $A^{+}$
into $GL_{2}^{+}(\mathbb{R})^{2}$. We denote by 
\[
\rho=(\rho_{1},\rho_{2}):A\rightarrow M_{2}(\mathbb{R})^{2}
\]
this representation of $A$. Note that these $\rho_{i},\, i=1,2$
are extensions of two morphisms $k\rightarrow\mathbb{R}$ (where $\mathbb{R}\subset M_{2}(\mathbb{R})$
is identified with diagonal matrices) corresponding to the places
$\sigma_{1},\sigma_{2}$.

Let us denote by $\mathcal{O}$ and $\mathcal{O}_{k}$ a maximal order
of $A$ and the ring of integers of $k$. Let $\mathcal{O}^{*}$ denote
the group of units of $\mathcal{O}$, $\mathcal{O}^{+}$ the group
of units in $\mathcal{O}$ with totally positive reduced norm and
$\mathcal{O}^{1}\subset\mathcal{O}$ the group of units of reduced
norm $1$. The group $\mathcal{O}^{1}$ is via $\rho$ a discrete
subgroup of $SL_{2}(\mathbb{R})^{2}$, whereas $\mathcal{O}^{+}$
is embedded as a discrete subgroup in $GL_{2}^{+}(\mathbb{R})^{2}$.

The group $A^{+}$ and any subgroup $G\subset A^{+}$ acts on the
product $\mathbb{H}\times\mathbb{H}$ of two copies of the upper half
plane as follows: if $(z,w)\in\HH\times\HH$ is a point and $g\in G$,
then $\rho(g)$ is represented by two matrices $g_{i}=\rho_{i}(g)=\left(\begin{array}{cc}
a_{i} & b_{i}\\
c_{i} & d_{i}
\end{array}\right)$ ($i=1,2$) and: 
\[
g(z,w)=(g_{1}z,g_{2}w):=(\frac{a_{1}z+b_{1}}{c_{1}z+d_{1}},\frac{a_{2}w+b_{2}}{c_{2}w+d_{2}}),
\]
 where $a_{1},a_{2}$ (resp. $(b_{1},b_{2})\ldots$) are conjugates
with respect to the places $\sigma_{1},\sigma_{2}$ over the Galois
closure of $k$ in $\mathbb{R}$. The action of $G$ is not effective;
the center $Z(G)$ acts trivially on $\mathbb{H}\times\mathbb{H}$
and therefore we will consider subgroups $\Gamma=G/Z(G)\subset A^{+}/k^{\ast}$
rather than subgroups $G\subset A^{+}$. Let us write $\Gamma_{\mathcal{O}}(1)$
or sometimes simply $\Gamma(1)$ for the group $\mathcal{O}^{1}/\{\pm1\}$.
The group $\Gamma(1)$ and also any subgroup $\Gamma$ of $A^{+}/k^{\ast}$
commensurable with $\Gamma(1)$ acts properly discontinuously on $\mathbb{H}\times\mathbb{H}$.

The quotient $X_{\Gamma}=\mathbb{H}\times\mathbb{H}/\Gamma$ is a
compact algebraic surface, which will be called \emph{quaternionic
Shimura surface} (corresponding to $\Gamma$) in the sequel.

\subsection{The involution exchanging factors and involutions of the second kind.}

Let, as above, $k$ be a totally real number field and consider the
quaternion algebra $A=A(k,\mathfrak{p}_{1},\dots,\mathfrak{p}_{r})$
over $k$ unramified at two infinite places of $k$ and ramified at
all the other infinite places of $k$ and at the finite places $\mathfrak{p}_{1},\ldots,\mathfrak{p}_{r}$.
\begin{defn}
An \emph{involution of second kind} on $A$ is a map $\tau:A\to A$
such that $\tau^{2}(a)=a$, $\tau(a+b)=\tau(a)+\tau(b)$, $\tau(ab)=\tau(a)\tau(b)$
for all $a,b\in A$ and such that the restriction of $\tau$ to $k$
is a non-trivial automorphism of $k$. 
\end{defn}
Let $\ell=k^{\tau}$ be the fixed field of $\tau$. Then $k/\ell$
is a quadratic extension and in this case we will say that $\tau$
is a \emph{$k/\ell$-involution}. 
\begin{rem}
(See also \cite[Lemma 4.2]{Granath}). Let $\tau$ and $\sigma$ be
two $k/\ell$-involutions of second kind on $A$. Then there exists
$m\in A^{*}$ such that $\sigma(a)=m^{-1}\tau(a)m$ and $\tau(m)^{*}=m$,
where $a\to a^{*}$ is the canonical anti-involution on $A$, that
is, the uniquely determined anti-involution $\ast:A\to A$ such that
the reduced norm is $Nrd(x)=xx^{*}$ and the reduced trace is $Trd(x)=x+x^{*}$.

Following \cite{Granath} we choose to work with involutions on $A$
which are, by above definition, particularly ring homomorphisms of
$A$. In the literature more often one works with anti-involutions,
that is, with maps $\rho:A\to A$ satisfying $\rho(ab)=\rho(b)\rho(a)$.
These two kinds of maps are linked by the canonical anti-involution.
First observe that every involution of second kind commutes with the
canonical anti-involution, that is, for every $a\in A$ we have $\tau(a)^{*}=\tau(a^{*})$,
see \cite{Granath}. From this it follows that there is one-to-one
correspondence between involutions and anti-involutions on $A$ given
by the rule 
\[
\tau\mapsto\tau^{*}\ \text{and}\ \rho\mapsto\rho^{*}
\]
where $\tau^{*}$, resp.~$\rho^{*}$, is defined as $\tau^{*}(a)=\tau(a)^{*}$,
resp.~$\rho^{*}(a)=\rho(a)^{*}$, for every $a\in A$. In this way,
involutions of second kind correspond to classically studied anti-involutions
of second kind.
\end{rem}
A criterion for the existence of anti-involutions of second kind is
well-known and goes back to work of Albert and Landherr. 
\begin{prop}
\label{existence_type_2}(See also \cite[Lemma 4.3]{Granath} and
more generally \cite[Theorem 3]{Landherr}). Let $k/\ell$ be a quadratic
extension of totally real fields. Let $\alpha$ be the non-trivial
$\ell$-automorphism of the extension $k/\ell$ and let $A=A(k,\mathfrak{p}_{1},\dots,\mathfrak{p}_{r})$.
Then, there exists a $k/\ell$-involution $\tau$ of second kind on
$A$ (i.~e.~$\tau(xa)=x^{\alpha}\tau(a)\ \text{for all}\ x\in k,a\in A$)
if and only if 
\begin{enumerate}
\item $r$ is even and after a suitable renumbering of the $\mathfrak{p}_{i}$,
we have $\mathfrak{p}_{2i-1}=\mathfrak{p}_{2i}^{\alpha}$ and $\mathfrak{p}_{i}\neq\mathfrak{p}_{i}^{\alpha}$
($i=1,\ldots,r/2$). 
\item With the notations and hypothesis of section \ref{sub:Notations,-quaternionic-Shimura},
it holds that $\sigma_{2}=\sigma_{1}\circ\alpha$. 
\end{enumerate}
\end{prop}
\begin{proof}
Let us first show that there exists an $k/\ell$-involution $\tau$
on $A$ if and only if there exists a quaternion subalgebra $A'\subset A$
whose center is $\ell$.\\
 Namely, if such $A'$ exists then $A'\otimes_{\ell}k=A$ and $a\otimes x\stackrel{\tau}{\mapsto}a\otimes x^{\alpha}$
is a $k/\ell$-involution. Conversely, let $A'=\{a\in A\mid\tau(a)=a\}$.
This is a $\ell$-subalgebra of $A$. Let $\theta\in k$ be such that
$k=\ell(\theta)$ and $\tau(\theta)=-\theta$. We claim that $A=A'\oplus\theta A'$.
To see this we write an element $a\in A$ as $a=\frac{{1}}{2}(a+\tau(a))+\frac{1}{2}\theta\frac{a-\tau(a)}{\theta}$.
It follows that $A=A'+\theta A'$. Clearly, an element $a\in A'\cap\theta A'$
satisfies $\tau(a)=a=-a,$ hence $a=0$ and it follows that $kA'\cong k\otimes_{\ell}A'=A$
and therefore $A'$ is a quaternion algebra over $\ell$ since $A=A'\otimes_{\ell}k$
is a quaternion algebra.

Now we compare the sets of ramification of $A$ and $A'$. If $v$
is a place of $\ell$ such that $A'\otimes_{\ell}\ell_{v}\cong M_{2}(\ell_{v})$,
then for any place $w$ of $k$ lying over $v$ we have $A\otimes_{k}k_{w}\cong M_{2}(k_{w})$,
so if $A'$ is unramified at a place $v$ of $\ell$, then $A$ is
unramified at every place $w$ of $k$ lying over $v$. Particularly,
since $k$ and $\ell$ are totally real, if $v$ is an infinite place
of $\ell$, that is, an embedding $v:\ell\hookrightarrow\RR$ then
there are exactly two places of $k$ lying over $v$, that is, two
embeddings $w_{1},w_{2}:k\hookrightarrow\RR$ extending $v$ and satisfying
$w_{2}=w_{1}\circ\alpha$.

Assume now that $v$ is a place of $\ell$ where $A'$ is ramified.
If $v$ is an infinite place of $\ell$, then, again since $k$ and
$\ell$ are both totally real, $A$ is ramified at every place $w$
of $k$ over $v$. Assume that $v$ is a finite place of $\ell$ corresponding
to a prime ideal $\mathcal{P}$. There are two different possibilities:
$\mathcal{P}$ is split in $k$, that is, $\mathcal{P}\mathcal{O}_{k}=\mathfrak{p}\mathfrak{p}^{\alpha}$,
with two prime ideals $\mathfrak{p}\neq\mathfrak{p}^{\alpha}$ of
$k$. As $k_{\mathfrak{p}}=\ell_{\mathcal{P}}$, $A'$ is ramified
at $\mathfrak{\mathcal{P}}$ if and only if $A$ is ramified at $\mathfrak{p}$.
If $\mathcal{P}$ is non-split in $k$, i.~e., if only one prime
$\mathfrak{p}$ is lying over $\mathcal{P}$, then by the theorem
of Hasse, the field $k_{\mathfrak{p}}$ can be embedded into $A'\otimes_{\ell}\ell_{\mathcal{P}}$
and therefore $A'\otimes_{\ell}\ell_{\mathcal{P}}\cong M_{2}(k_{\mathfrak{p}})$.
In this case $A$ is unramified at $\mathfrak{p}$. It follows, that
the existence of a $k/\ell$-involution implies the conditions (1)
and (2).

Conversely, choosing the set of ramified places according to (1) and
(2), we can construct a quaternion algebra $A'$ over $\ell$ such
that $A'\otimes_{\ell}k$ and $A$ are ramified exactly at the same
set of primes. This implies an isomorphism $A\cong A'\otimes_{\ell}k$
and thus the existence of an $k/\ell$-involution of second kind \end{proof}
\begin{rem}
Let us note that with notations of section \ref{sub:Notations,-quaternionic-Shimura}
the above proposition implies that in the case of the existence of
a $k/\ell$-involution $\tau$ on $A$ we particularly have the relation
$\rho_{2}=\rho_{1}\circ\tau$ on $A$. Note also that the condition
(2) is superfluous in the case where $k$ is a real quadratic field,
since there is only one choice of $\sigma_{2}=\sigma_{1}\circ\alpha$. 
\end{rem}
Let $\tau$ be an involution of second kind on $A$. Let $\Gamma\subset A^{+}/k^{\ast}$
be a subgroup stable by $\tau$, that is, for all $\gamma\in\Gamma$
we have $\tau(\gamma)\in\Gamma$. Since $\rho_{2}=\rho_{1}\circ\tau$,
the images of $\Gamma$ in $PSL_{2}(\mathbb{R})$ by $\rho_{1}$ and
$\rho_{2}$ are the same, and since this image is isomorphic to $\Gamma$
we denote it also by $\Gamma$. In order to avoid more long-winded
notations, let us write for an involution of second kind $\tau$ shortly
\[
\tau(a)=\overline{a}\ \text{for all}\ a\in A.
\]
In particular, the non-trivial $\ell$-automorphism of $k$ is denoted
by $x\mapsto\overline{x}$ (for $x\in k$). Identifying $k$ with
$\sigma_{1}(k)$, say, the action of any element $\gamma\in A^{+}/k^{\ast}$
on $\mathbb{H}\times\mathbb{H}$ given by
\[
\gamma(z_{1},z_{2})=(\rho_{1}(\gamma)z_{1},\rho_{2}(\gamma)z_{2})
\]
 is written as
\[
\gamma(z_{1},z_{2})=(\gamma z_{1},\bar{\gamma}z_{2}).
\]

Let $\mu:\mathbb{H}\times\mathbb{H}\rightarrow\mathbb{H}\times\mathbb{H}$
be the involution that exchanges the two factors. The group $Aut(\mathbb{H}\times\mathbb{H})$
is the semi direct product of $Aut(\mathbb{H})\times Aut(\mathbb{H})$
and the group $\mathbb{Z}/2\mathbb{Z}$ generated by $\mu$. 
\begin{prop}
\label{pro:involution second kind implies involution} Let $\Gamma$
be a subgroup of $A^{+}$ commensurable with $\mathcal{O}^{1}$ for
some maximal order $\mathcal{O}$ in $A$. Suppose that there exists
an involution $\tau$ of second kind on $A$ preserving $\Gamma$.
Then, the automorphism $\mu$ of $\mathbb{H}\times\mathbb{H}$ induces
an involution $\sigma$ on the surface $X{}_{\Gamma}=\mathbb{H}\times\mathbb{H}/\Gamma$. \end{prop}
\begin{proof}
We need to prove that $\mu$ sends an orbit under the action of $\Gamma$
to another $\Gamma$-orbit. We have:
\[
\Gamma(z_{1},z_{2})=\{(\gamma z_{1},\bar{\gamma}z_{2})|\,\gamma\in\Gamma\},
\]
 thus
\[
\mu(\Gamma(z_{1},z_{2}))=\{(\bar{\gamma}z_{2},\gamma z_{1})|\,\gamma\in\Gamma\}.
\]
 Since $\tau:a\to\bar{a}$ preserves $\Gamma$, we get:
\[
\mu(\Gamma(z_{1},z_{2}))=\{(\gamma z_{2},\bar{\gamma}z_{1})|\,\gamma\in\Gamma\}=\Gamma(z_{2},z_{1}),
\]
 therefore $\mu(\Gamma(z_{1},z_{2}))=\Gamma(z_{2},z_{1})$ is an orbit
under the action of $\Gamma$ on $\mathbb{H}\times\mathbb{H}$ and
$\mu$ acts on the orbit space $X_{\Gamma}$. 
\end{proof}

\subsection{Automorphisms of Shimura surfaces.}

Let $\mu\in\aut\mathbb{H}\times\mathbb{H}$ be the involution that
exchanges the two factors. Let $X_{\Gamma}=\mathbb{H}\times\mathbb{H}/\Gamma$
be a smooth quaternionic Shimura surface. By \cite[Theorem 3.12]{DzambicRoulleau}
and its proof, we get: 
\begin{prop}
\label{pro:fixed point of dimension 1}Suppose that the fixed locus
$C$ of an automorphism $\sigma$ of $X=X_{\Gamma}$ contains an one-dimensional
component. Then $\sigma$ is an involution that lifts on the universal
cover to (a conjugate of) $\mu$. 
\end{prop}
Let us consider an involution $a\mapsto\bar{a}$ on $A$ of second
kind and a torsion-free group $\Gamma$ commensurable with a group
$\Gamma(1)$, and as above stable under $a\mapsto\overline{a}$. Recall
that for $\gamma\in\Gamma$ and $t=(z_{1},z_{2})\in\mathbb{H}\times\mathbb{H}$
the action of $\gamma$ on $\mathbb{H}\times\mathbb{H}$ is given
by 
\[
\gamma t=(\gamma z_{1},\bar{\gamma}z_{2}).
\]
 The image of $t$ on $X_{\Gamma}=\mathbb{H}\times\mathbb{H}/\Gamma$
is a fixed point of the involution $\sigma$ induced by $\mu$ if
and only if 
\[
\Gamma(z_{1},z_{2})=\Gamma(z_{2},z_{1}),
\]
 which is the case if and only if there exists $\gamma$ in $\Gamma$
such that 
\[
(z_{1},z_{2})=\gamma(z_{2},z_{1})=(\gamma z_{2},\bar{\gamma}z_{1}),
\]
 that is, if and only if $z_{1}=\gamma z_{2}$ and $z_{2}=\bar{\gamma}\gamma z_{2}$.
Since $\Gamma$ is torsion-free (because $X$ is smooth), the image
of $t$ in $X$ is a fixed point of $\sigma$ if and only if $\bar{\gamma}\gamma=1$.

\textcolor{black}{As in \cite{Granath}, for any $\beta\in A\setminus\{0\}$,
let $\Delta_{\beta}$ be the disk $\Delta_{\beta}=\{(z,\beta z)/z\in\mathbb{H}\}$}.
\textcolor{black}{We have $\lambda\Delta_{\beta}=\Delta_{\bar{\lambda}\beta\lambda^{-1}}$
for any $\lambda\in A\setminus\{0\}$.} We denote by $F_{\beta}$
the image of $\Delta_{\beta}$ in $X_{\Gamma}$.

Let $\gamma\in\Gamma$ be such that $\bar{\gamma}=\gamma^{-1}$. Since
$\bar{\gamma}\gamma=1$, we have $(\gamma z,z)=\gamma(z,\gamma z)$,
therefore for any point $t$ of $\Delta_{\gamma}$, we obtain 
\[
\mu\Gamma t=\Gamma\mu t=\Gamma\gamma t=\Gamma t,
\]
 and the image $F_{\gamma}$ of $\Delta_{\gamma}$ on $X_{\Gamma}$
is fixed by the involution $\sigma$. That implies that $F_{\gamma}$
is a smooth irreducible algebraic curve, more precisely a Shimura
curve, and that there is only a finite number of such curves. We obtain: 
\begin{cor}
The fixed point set of $\sigma$ is the union of smooth disjoint Shimura
curves $F_{\gamma}$ with $\gamma\in\Gamma$ such that $\gamma\bar{\gamma}=1$. \end{cor}
\begin{rem}
Since the irreducible components of the fixed locus $C$ are smooth
disjoint Shimura curves on the Shimura surface $X$, we get by the
Hirzebruch Proportionality Theorem 
\[
2C^{2}=-K_{X}C=4(1-g)
\]
where $g>1$ is the arithmetic genus of $C$. If the irreducible componants
are $C_{j},\, j=1\cdots,k$ of genus $g_{j}$, then $C_{j}^{2}=2(1-g_{j})$
and $C^{2}=\sum C_{j}^{2}=\sum2(1-g_{j})$. Moreover 
\[
K_{X}C=\sum4(g_{j}-1)
\]
thus $C^{2}+K_{X}C=\sum2g_{i}-2=2g-2$ and $C^{2}=2-2g$.
\end{rem}
Recall that \textcolor{black}{$\lambda C_{\beta}=C_{\bar{\lambda}\beta\lambda^{-1}}$
for any $\lambda\in A\setminus\{0\}$}, and in particular for every
$\lambda\in\Gamma$ we have: 
\[
F_{1}=F_{\bar{\lambda}\lambda^{-1}}.
\]
Of course, for $\alpha=\bar{\lambda}\lambda^{-1}$ we have $\alpha\bar{\alpha}=1$.
We conjecture that for a smooth surface $X$ the fixed locus of such
$\sigma$ is irreducible. As we see immediatly, this is equivalent
to the following: 
\begin{conjecture}
\label{conjecture1}Let be $\lambda\in\Gamma$ such that $\lambda\bar{\lambda}=1$.
There exists $\gamma\in\Gamma$ such that $\lambda=\bar{\gamma}\gamma^{-1}$. 
\end{conjecture}

\section{Quotient of a quaternionic Shimura surface by an involution of second
kind.}

\label{sec:Quotient and fundamental group}

\subsection{Invariants of the quotient}

Let $\Gamma$ be a lattice preserved by an involution of the second
kind and let $\sigma$ be the corresponding involution acting on the
Shimura surface $X=X_{\Gamma}$. Let $C$ be the smooth curve of arithmetic
genus $g$ fixed by the involution. 
\begin{prop}
\label{invariants_of_quotient} The quotient surface $Z=X/\sigma$
is smooth and has invariants:
\[
K_{Z}^{2}=e(X)+5(1-g),\, c_{2}=\frac{1}{2}e(X)+1-g,\, p_{g}=\frac{e(X)-4-4g}{8},\, q=0,
\]
 where $e(X)=c_{2}(X)$ is the topological Euler number of $X$. \\
 If $(K_{X}-C)^{2}>0$, then $Z$ has general type; this condition
on the positivity is satisfied if $e(X)\le36$.\\
 Suppose $e(X)=12$, then $C$ is irreducible of genus $g=2$ and
the quotient surface $X/\sigma$ has invariants:
\[
K_{X}^{2}=7,\, c_{2}=5,\, p_{g}=0,\, q=0.
\]
 Suppose $e(X)=16$, then $C$ is irreducible of genus $g=3$ and
the quotient surface $Z$ has invariants:
\[
K_{X}^{2}=6,\, c_{2}=6,\, p_{g}=0,\, q=0.
\]
\end{prop}
\begin{proof}
Let $\pi:X\to Z=X/\sigma$ be the quotient map. Since $K_{X}^{2}=2e$
(for $e=c_{2}(X)$), $K_{X}=\pi^{*}K_{Z}+C$ and $C^{2}=2(1-g),\, K_{X}C=4(g-1)$
(here we use that $C$ is a disjoint union of smooth Shimura curves)
we get
\[
K_{Z}^{2}=\frac{1}{2}(K_{X}-C)^{2}=\frac{1}{2}(K_{X}^{2}-2K_{X}C+C{}^{2})=e-5(g-1).
\]
 Moreover by general formulas on quotient surfaces (see e.g. \cite{DzambicRoulleau})
\[
e(Z)=\frac{1}{2}(e(X)+e(C))=\frac{1}{2}e(X)+1-g.
\]
 As $q(X)=0$, we get $q(Z)=0$ and
\[
p_{g}(Z)=\chi(\mathcal{O}_{Z})-1=\frac{e(X)-4(g+1)}{8}.
\]
 As $g\geq2$ because $X$ is hyperbolic and $p_{g}(Z)\geq0$, we
get that
\[
2\leq g\leq\frac{e(X)-4}{4},
\]
 thus
\[
e(Z)\geq\frac{e(X)}{4}+2,\,\chi(\mathcal{O}_{Z})\geq1
\]
 and $K_{Z}^{2}\geq10-\frac{e}{4}.$ Let us prove that $Z$ has general
type if $K_{Z}^{2}>0$. Since $\pi^{*}K_{Z}=K_{X}-C$, it is enough
to prove that powers $\mathcal{L}^{m}$ of $\mathcal{L}=\mathcal{O}_{X}(K_{X}-C)$
have sections growing in $c\cdot m^{2}$, where $c>0$ is a constant.
Suppose that $K_{Z}^{2}=\frac{1}{2}(K_{X}-C)^{2}>0$ (note that by
the preceding discussion, this condition is always satisfied for a
surface $X$ with $e(X)<40$). By the Riemann-Roch Theorem, we have
\[
\chi(\mathcal{L}^{m})=\frac{m^{2}}{2}(K_{X}-C)^{2}-\frac{m}{2}K_{X}(K_{X}-C)+\chi(\mathcal{O}_{X}).
\]
 Serre duality gives
\[
\chi(\mathcal{L}^{m})=H^{0}(X,\mathcal{L}^{m})-H^{1}(X,\mathcal{L}^{m})+H^{0}(X,mC-(m-1)K_{X}).
\]
 Suppose that $D=mC-(m-1)K_{X}$ is effective. As $K_{X}$ is ample
$K_{X}D>0$. But
\[
K_{X}D=m(4g-4-2e)+2e
\]
 and as $g\leq\frac{e(X)-4}{4}$, we get $4g-4\leq e-8$ and
\[
K_{X}D\leq m(-8-e)+2e<0
\]
 for $m\geq3$. Therefore $H^{0}(X,mC-(m-1)K_{X})=0$ and $Z$ has
general type. 
\end{proof}
The possibilities for values $12,16,\dots,36$ of $e(X)$ and for
the genus $g$ of $C$ are listed in the above tables:

\vspace{0.1in}

\begin{tabular}{|c|c|c|c|c|}
\hline 
$e(X)$  & $g$  & $K_{Z}^{2}$  & $c_{2}(Z)$  & $p_{g}(Z)$\tabularnewline
\hline 
$e=12$  & $2$  & $7$  & $5$  & $0$\tabularnewline
\hline 
$e=16$  & $3$  & $6$  & $6$  & $0$\tabularnewline
\hline 
$e=20$  & $2$  & $15$  & $9$  & $1$\tabularnewline
\cline{2-5} 
 & $4$  & $5$  & $7$  & $0$\tabularnewline
\hline 
$e=24$  & $3$  & $14$  & $10$  & $1$\tabularnewline
\cline{2-5} 
  & $5$  & $4$  & $8$  & $0$\tabularnewline
\hline 
$e=28$  & $2$  & $23$  & $13$  & $2$\tabularnewline
\cline{2-5} 
 & $4$  & $13$  & $11$  & $1$\tabularnewline
\hline 
\end{tabular}\hspace{0.1in}%
\begin{tabular}{|c|c|c|c|c|}
\hline 
$e(X)$  & $g$  & $K_{Z}^{2}$  & $c_{2}(Z)$  & $p_{g}(Z)$\tabularnewline
\hline 
$e=28$  & $6$  & $3$  & $9$  & $0$\tabularnewline
\hline 
 & $3$  & $22$  & $14$  & $2$\tabularnewline
\cline{2-5} 
$e=32$  & $5$  & $12$  & $12$  & $1$\tabularnewline
\cline{2-5} 
 & $7$  & $2$  & $10$  & $0$\tabularnewline
\hline 
 & $2$  & $31$  & $17$  & $3$\tabularnewline
\cline{2-5} 
 $e=36$  & $4$  & $21$  & $15$  & $2$\tabularnewline
\cline{2-5} 
 & $6$  & $11$  & $13$  & $1$\tabularnewline
\cline{2-5} 
 & $8$  & $1$  & $11$  & $0$\tabularnewline
\hline 
\end{tabular}

\subsection{The fundamental group of the quotient}

Let us recall some results about fundamental groups. Let $G$ be a
discontinuous group of homeomorphisms of a path connected, simply
connected, locally compact metric space $M$, and let $G_{tor}$ be
the normal subgroup of $G$ generated by those elements which have
fixed points, or equivalently, the torsion elements. Then 
\begin{thm}
\cite{Armstrong}. The fundamental group of the orbit space $M/G$
is isomorphic to the factor group $G/G_{tor}$. 
\end{thm}
Let $X=X_{\Gamma}$ be our Shimura surface with fundamental group
$\Gamma$ such that the involution $\mu$ switching the two factors
of the bi-disk $\mathbb{H}\times\mathbb{H}$ acts on $X$ by an involution
denoted by $\sigma$. The fundamental group of the quotient surface
$X/\sigma$ is isomorphic to $\Gamma'/\Gamma'_{tor}$, where $\Gamma'$
is the group generated by $\Gamma$ and $\mu$.

Recall that a group $G$ acting on the space $M$ is discontinuous
if: \\
 (1) the stabilizer of each point of $M$ is finite, and\\
 (2) each point $x\in M$ has a neighborhood $U$ such that any element
of $G$ not in the stabilizer of $x$ maps $U$ outside itself (i.e.
if for $g\in G$, $gx\not=x$ then $U\cap gU$ is empty). 
\begin{lem}
The group $\Gamma'$ is discontinuous. \end{lem}
\begin{proof}
Since $M=\mathbb{H}\times\mathbb{H}$ is a locally compact Hausdorff
space, $\Gamma'$ is discontinuous if and only if it is discrete subgroup
of $Aut(\mathbb{H}\times\mathbb{H})$. The latter assertion follows
from the fact that $\Gamma'$ is an index-2 extension of the discrete
group $\Gamma$. 
\end{proof}
For any $\gamma\in\Gamma$, we have $\mu\gamma=\bar{\gamma}\mu.$
Let $g\in\Gamma'_{tor}$ be a non-trivial torsion element. Since $\Gamma$
is torsion free, $g\not\in\Gamma$ and there exists $\lambda\in\Gamma$
such that $g=\lambda\mu$. The order of $g$ is then divisible by
$2$ and we have $g^{2n}=(\lambda\bar{\lambda})^{n}$. Since $\Gamma$
is torsion free, $g^{2n}=1$ if and only if $\lambda\bar{\lambda}=1$.
Therefore a torsion element $g$ of $\Gamma'$ has order $2$ and
there exists $\lambda\in\Gamma$ such that $g=\lambda\mu$ and $\lambda\bar{\lambda}=1$.

As an immediate consequence we obtain: 
\begin{lem}
The fundamental group $\Gamma'/\Gamma'_{tor}$ of $X/\sigma$ is isomorphic
to the group $\Gamma/N$ where $N$ is the normal group generated
by the $\lambda\in\Gamma$ such that $\lambda\bar{\lambda}=1$. 
\end{lem}
Since for any $\gamma\in\Gamma$, the group $\Gamma'_{tor}$ contains
$\bar{\gamma}\mu\bar{\gamma}^{-1}=\bar{\gamma}\gamma^{-1}\mu$, we
see that $N$ contains $\bar{\gamma}\gamma^{-1}$ for every $\gamma\in\Gamma$,
therefore the quotient $\Gamma/N$ forces the relation $\gamma=\bar{\gamma}$
for any $\gamma\in\Gamma$.

Let us denote by $H$ the normal subgroup of $\Gamma$ generated by
the elements $\bar{\gamma}\gamma^{-1}$, $\gamma\in\Gamma$. Note
that under Conjecture \ref{conjecture1}, the group $H$ is equal
to $N$. The group $\pi_{1}(X/\sigma)=\Gamma'/\Gamma'_{tor}\simeq\Gamma/N$
is a quotient of $\Gamma/H$.
\begin{thm}
Suppose that $\Gamma$ is a subgroup of $\Gamma(1)$. Then $\Gamma/H$
is a finite group and the fundamental group of $X/\sigma$ is finite.\end{thm}
\begin{proof}
Let $A'$ be a quaternion algebra over the field $\ell$ as above
such that $A=A'\otimes k$ and the involution of second kind on $A$
is given by $a'\otimes u\to a'\otimes\bar{u}$. Let $k'$ be a degree
$2$ extension of $\ell$ such that $A'\otimes_{\ell}k'=M_{2}(k')$
and let $K$ be the compositum of $k,k'$ : $K=k\otimes_{\ell}k'$.
Then the algebra $A\otimes_{k}K$ is $A'\otimes_{\ell}K=M_{2}(k')\otimes_{\ell}k=M_{2}(K)$.
The involution of second kind $a\to\bar{a}$ extends to $M_{2}(K)$
and acts on each entries fixing $M_{2}(k')\subset M_{2}(K)$. The
embedding 
\[
j:A\hookrightarrow M_{2}(K)
\]
is equivariant for the action of the involution: $\forall a\in A$,
$j(\bar{a})=\overline{j(a)}$, where the action on the left hand side
is the conjugation on each entries of the matrix. 

The group $j(\Gamma(1))$ is a subgroup of $PSL_{2}(\mathcal{O}_{K})$.
Let $I\subset\mathcal{O}_{K}$ be the (non-trivial) ideal generated
by the elements $\bar{a}-a,\, a\in\mathcal{O}_{K}$. The ring $O_{K}/I$
is a finite ring therefore the subgroup $\Gamma/H$ of $PSL_{2}(\mathcal{O}_{K}/I)$
is finite.

The fundamental group of $X/\sigma$ is (isomorphic to) $\Gamma/N$
which is a quotient of the finite group $\Gamma/H$, therefore $\pi_{1}(X/\sigma)$
is finite. 
\end{proof}

\section{Examples}

\label{sec:Examples}

\subsection{Aim and terminology}

Our goal is to find examples of smooth quaternionic Shimura surfaces
$X_{\Gamma}$ together with an involution $\sigma$ on $X_{\Gamma}$
having one-dimensional fixed locus. So, we consider an indefinite
quaternion algebra $A$ over a totally real field $k$ of degree $n=[k:\QQ]$,
unramified exactly at two infinite places of $k$ and consider groups
$\Gamma$ commensurable with $\Gamma_{\OO}(1)$, the group of norm-1
elements of a maximal order $\OO\subset A$ (modulo center). 
\begin{defn}
Let $k,A,\OO$ be as above. We say that a discrete group $\Gamma$
in the commensurability class of $\Gamma_{\OO}(1)$ is \emph{admissible
of type $e$} if: 
\begin{enumerate}
\item $\Gamma$ is torsion-free. 
\item $e(X_{\Gamma})=e$ where $e(X_{\Gamma})$ is denoting the (orbifold-)
Euler number. 
\item On $A$ there exists an involution $\tau$ of second kind such that
$\Gamma$ is invariant under $\tau$. 
\end{enumerate}
\end{defn}
Let us remark that according to Proposition \ref{invariants_of_quotient},
the quotient $X_{\Gamma}/\sigma$ will be of general type if $e(X_{\Gamma})\leq36$.
Hence, we will focus on admissible groups of type $e=12,16,20,...,36$.
Also, the proposition \ref{existence_type_2} gives us a condition
which guarantees the existence of an involution of second kind on
$A$.

\subsection{Smoothness and the Euler number\label{sub:Smoothness-and-the-Euler}}

Let $A=A(k,\mathfrak{p}_{1},...,\mathfrak{p}_{2m})$ be as above and
assume that there exists a $k/\ell$-involution on $A$ with respect
to a subfield $\ell\subset k$. According to Proposition \ref{existence_type_2}
we particularly assume that the primes $\mathfrak{p}_{i},\, i=1,\dots,2m$
in $\mathcal{O}_{k}$ come in pairs: there exist primes $\mathcal{P}_{1}$,...,$\mathcal{P}_{m}$
of $\mathcal{O}_{\ell}$ such that $\mathfrak{p}_{1}\mathfrak{p}_{2}=(\mathcal{P}_{1})$,...,$\mathfrak{p}_{2m-1}\mathfrak{p}_{2m}=(\mathcal{P}_{m})$.

If $\Gamma$ is commensurable with $\Gamma_{\OO}(1)$, we have the
following general formula for the orbifold Euler number of the Shimura
surface $X_{\Gamma}$ \textcolor{red}{}
\begin{prop}
(see \cite{shimizu}, \cite{vign}) \label{shimizu} Let $k$ and
$A=A(k,\mathfrak{p}_{1},...,\mathfrak{p}_{2m})$ be as above. Assume
that there exists a $k/\ell$-involution on $A$. Let $n=[k:\QQ]$,
$\zeta_{k}(\ )$ be the Dedekind zeta function of $k$ and $d_{k}$
denote the discriminant of $k$. Then the orbifold Euler number of
$X_{\Gamma}$ equals
\[
e(X_{\Gamma})=[\Gamma_{\OO}(1):\Gamma]\cdot\frac{d_{k}^{3/2}\zeta_{k}(2)}{2^{2n-3}\pi^{2n}}\prod_{i=1}^{m}(N\mathcal{P}_{i}-1)^{2},
\]
 where $N\mathcal{P}=|\mathcal{\OO}_{\ell}/\mathcal{P}|$ denotes
the norm of $\mathcal{P}$ and where
\[
[\Gamma_{\OO}(1):\Gamma]=\frac{[\Gamma_{\OO}(1):\Gamma\cap\Gamma_{\OO}(1)]}{[\Gamma:\Gamma\cap\Gamma_{\OO}(1)]}
\]
 is the generalized index of two commensurable groups. 
\end{prop}
When $k=\QQ(\sqrt{d})$ is a real quadratic field, we have a particularly
handable formula $\zeta_{k}(2)=\frac{\pi^{4}B_{2,k}}{6d_{k}^{3/2}}$
for the value $\zeta_{k}(2)$ in terms of the second generalized Bernoulli
number. This implies 
\begin{cor}
(see \cite{Shavel78}) Let $k=\QQ(\sqrt{d})$ be a real quadratic
field. We denote by $B_{2,k}$ the second generalized Bernoulli number
associated with the quadratic Dirichlet character $\chi_{k}(p)=\left(\frac{d_{k}}{p}\right)$
of $k$. Then,
\[
e(X_{\Gamma})=[\Gamma_{\OO}(1):\Gamma]\cdot\frac{B_{2,k}}{12}\prod_{i=1}^{m}(p_{i}-1)^{2},
\]
 where $p_{1},\ldots,p_{m}$ are rational primes such that $\mathfrak{p}_{1}\mathfrak{p}_{2}=(p_{1})$,...,$\mathfrak{p}_{2m-1}\mathfrak{p}_{2m}=(p_{m})$. 
\end{cor}
As next, we would like to discuss the question about the smoothness
of $X_{\Gamma}$. Note that $X_{\Gamma}$ is smooth if and only if
$\Gamma$ is torsion-free. Here, we will concentrate on subgroups
$\Gamma\subset\Gamma_{\OO}(1)$. It is worth to recall that the torsions
in $\Gamma_{\OO}(1)$ correspond to ring embeddings $\OO_{k}[\xi_{n}]\longrightarrow\OO$
of the roots of unity into the maximal order $\OO$. The general criterion
for the existence of torsions in $\Gamma_{\OO}(1)$ is as follows. 
\begin{lem}
(see \cite{Shavel78}) \label{torsions} Let $\xi_{n}$ be a primitive
$n$-th root of unity. There exists an element of order $n$ in $\Gamma_{\OO}(1)$
if and only if: 
\begin{enumerate}
\item $\xi_{n}+\xi_{n}^{-1}\in k$ 
\item every ramified prime in $A$ is non-split in $k(\xi_{n})$ 
\end{enumerate}
\end{lem}
The above lemma already gives us a bound for the order of possible
torsion elements. Namely, for any $a\in A\setminus k$, the algebra
$k(a)$ is commutative subfield of $A$. Since $\dim_{k}A=4$, and
$a\notin k$, $\dim_{k}k(a)=2$ and $k(a)$ is a quadratic extension
of $k$. Assume now that $\xi$ is a primitive $n$-th root of unity
embedded in $\OO$, then as $L=k(\xi)\subset A$ is a quadratic extension
of $k$, we have $\varphi(n)\leq2[k:\QQ]$, since $\varphi(n)=[\QQ(\xi):\QQ]\leq[L:\QQ]\leq2[k:\QQ]$.\\

Above results provide us with criteria to test the conditions in the
definition of admissible groups. Let us state a classification result
in the case of a real quadratic field $k$. 
\begin{thm}
\label{list_real_quadratic_candidates} Let $k=\QQ(\sqrt{d})$ be
a real quadratic field and consider the totally indefinite quaternion
algebra $A=A(k,\pP_{1},\overline{\pP}_{1}\ldots,\pP_{m},\overline{\pP}_{m})$
over $k$ ramified at the prime ideals dividing rational primes $p_{1},\ldots,p_{m}$
which are split in $k$. If $\Gamma\subset\Gamma_{\OO}(1)$ is an
admissible group of type $e$, then the possibilities are as follows:
\end{thm}
\begin{center}
\begin{tabular}{|c|c|c|c|}
\hline 
type $e$  & $k$  & Ram$(A)$  & $[\Gamma_{\OO}(1):\Gamma]$\tabularnewline
\hline 
$12$  & $\QQ(\sqrt{17})$  & $\pP_{2},\overline{\pP}_{2}$  & $18$\tabularnewline
\hline 
$16$  & $\QQ(\sqrt{13})$  & $\pP_{3},\overline{\pP}_{3}$  & $12$\tabularnewline
\hline 
$16$  & $\QQ(\sqrt{17})$  & $\pP_{2},\overline{\pP}_{2}$  & $24$\tabularnewline
\hline 
$20$  & $\QQ(\sqrt{17})$  & $\pP_{2},\overline{\pP}_{2}$  & $30$\tabularnewline
\hline 
$24$  & $\QQ(\sqrt{13})$  & $\pP_{3},\overline{\pP}_{3}$  & $18$\tabularnewline
\hline 
$24$  & $\QQ(\sqrt{17})$  & $\pP_{2},\overline{\pP}_{2}$  & $36$\tabularnewline
\hline 
$24$  & $\QQ(\sqrt{2})$  & $\pP_{7},\overline{\pP}_{7}$  & $4$\tabularnewline
\hline 
$24$  & $\QQ(\sqrt{33})$  & $\pP_{2},\overline{\pP}_{2}$  & $12$\tabularnewline
\hline 
$28$  & $\QQ(\sqrt{17})$  & $\pP_{2},\overline{\pP}_{2}$  & $42$\tabularnewline
\hline 
$32$  & $\QQ(\sqrt{13})$  & $\pP_{3},\overline{\pP}_{3}$  & $24$\tabularnewline
\hline 
$32$  & $\QQ(\sqrt{17})$  & $\pP_{2},\overline{\pP}_{2}$  & $48$\tabularnewline
\hline 
$32$  & $\QQ(\sqrt{28})$  & $\pP_{3},\overline{\pP}_{3}$  & $6$\tabularnewline
\hline 
$36$  & $\QQ(\sqrt{17})$  & $\pP_{2},\overline{\pP}_{2}$  & $54$\tabularnewline
\hline 
$36$  & $\QQ(\sqrt{33})$  & $\pP_{2},\overline{\pP}_{2}$  & $18$\tabularnewline
\hline 
\end{tabular}
\par\end{center}

In order to prove this theorem we will need the following elementary
lemma.
\begin{lem}
\label{TdividesG:H}Let $G$ be an arbitrary group and $H\subset G$
a torsion-free subgroup of finite index. If $T\subset G$ is a finite
subgroup, then $|T|$ divides $[G:H]$. \end{lem}
\begin{proof}
Let $G/H=\{g_{1}H,\ldots,g_{I}H\}$ be a set of left cosets of $H$
in $G$. The group $T$ acts by left multiplication on this set. And
moreover this action is free. Otherwise, we would have $t\cdot g_{i}H=g_{i}H$
with some non-trivial $t\in T$ and consequently, $g_{i}^{-1}tg_{i}\in H$,
which is not possible, since $H$ is torsion-free and $t$ as well
as $g_{i}^{-1}tg_{i}$, is of finite order. By elementary group theory,
the length of any $T$-orbit on $G/H$ is the same as the order of
$|T|$. And since $G/H$ is the union of different $T$-orbits, $|G/H|$
is divisible by $|T|$. 
\end{proof}
\emph{Proof of Theorem \ref{list_real_quadratic_candidates}.} Recall
that we restrict the type $e$ of an admissible group to values $e=12+4k\leq36$.
If $\Gamma\subset\Gamma_{\OO}(1)$ is admissible of type $e$, then
\[
36\geq[\Gamma_{\OO}(1):\Gamma]\frac{B_{2,k}}{12}\prod_{i=1}^{m}(p_{i}-1)^{2}.
\]
 Since $[\Gamma_{\OO}(1):\Gamma]$ and $(p_{i}-1)^{2}$ are positive
we have the condition $36\geq\frac{B_{2,k}}{12}$. By \cite{Shavel78},
Proposition 3.2, the Bernoulli number $B_{2,k}$ is bounded below
by $3d_{k}^{3/2}/50$ and this implies the upper bound $d_{k}<373$
for the discriminant of $k$. Using the formula for the second generalized
Bernoulli number given in \cite{Shavel78}, p.~228, we can compute
all the values $B_{2,k}$ for $d_{k}<373$ easily with the help of
a computer. With the list of all these values we check the necessary
conditions: 
\begin{itemize}
\item $36\geq B_{2,k}/12$
\item $B_{2,k}\mid12\cdot e=144,192,240,\ldots432$ for integral $B_{2,k}$
( $\Leftrightarrow$ $d_{k}\neq5$) and with obvious modification
for $d_{k}=5$.
\item The square part of $12e/B_{2,k}$ is divisible by the product $\prod_{p}(p-1)^{2}$
where $p$ runs over subsets of rational primes which are split in
$k$ (note that $p=2$, if split in $k$, contributes the factor $1$
to the product). 
\end{itemize}
We obtain the following list of tuples satisfying all the conditions

\begin{center}
\begin{tabular}{|c|c|c|c|}
\hline 
$d_{k}$  & $B_{2,k}$  & $e$  & $12e/B_{2,k}=I\cdot\prod(p-1)^{2}$\tabularnewline
\hline 
$137$  & $192$  & $16$  & $1=(2-1)^{2}$\tabularnewline
\hline 
$113$  & $144$  & $12$  & $1=(2-1)^{2}$\tabularnewline
\hline 
$109$  & $108$  & $36$  & $4=(3-1)^{2}$\tabularnewline
\hline 
$105$  & $144$  & $12$  & $1=(2-1)^{2}$\tabularnewline
\hline 
$85$  & $72$  & $24$  & $1=(3-1)^{2}$\tabularnewline
\hline 
$40$  & $28$  & $28$  & $12=3\cdot(3-1)^{2}$\tabularnewline
\hline 
$37$  & $20$  & $20$  & $12=3\cdot(3-1)^{2}$\tabularnewline
\hline 
$33$  & $24$  & $24$  & $12\cdot(2-1)^{2}$\tabularnewline
\hline 
$29$  & $12$  & $16$  & $16=(5-1)^{2}$\tabularnewline
\hline 
$29$  & 12  & $32$  & $32=2\cdot(5-1)^{2}$\tabularnewline
\hline 
$29$  & 12  & $36$  & $36=(7-1)^{2}$\tabularnewline
\hline 
$28$  & $16$  & $16$  & $12=3\cdot(3-1)^{2}$\tabularnewline
\hline 
$28$  & 16  & $32$  & $24=6(3-1)^{2}$\tabularnewline
\hline 
$24$  & $12$  & $16$  & $16=(5-1)^{2}$\tabularnewline
\hline 
$24$  & 12  & $32$  & $32=2\cdot(5-1)^{2}$\tabularnewline
\hline 
$17$  & $8$  & $12+4k$, $k=0,\ldots,6$  & $e\cdot(2-1)^{2}$\tabularnewline
\hline 
$13$  & $4$  & $12+4k$, $k=0,\ldots,6$  & $(9+3k)\cdot(3-1)^{2}$\tabularnewline
\hline 
$8$  & $2$  & $12,24,36$  & $2\cdot(7-1)^{2}$,$4\cdot(7-1)^{2}$, $6\cdot(7-1)^{2}$\tabularnewline
\hline 
$5$  & $\frac{4}{5}$  & $20$  & $3\cdot(11-1)^{2}$\tabularnewline
\hline 
\end{tabular}
\par\end{center}

We observe (keeping also in mind the splitting behavior of $2$ in
$k$) that the set of ramified places in $A$ is determined by the
value $12e/B_{2,k}$ in the table. 

As next we identify those subgroups $\Gamma\subset\Gamma_{\OO}(1)$
which cannot be torsion-free. For this we use the two Lemmas \ref{torsions}
and \ref{TdividesG:H}. Namely, note first that $\Gamma_{\OO}(1)$
contains at most torsions of order $2,3$ and $6$ for $k\neq\QQ(\sqrt{5}),\QQ(\sqrt{2})$
and additionally elements of order $5$ for $k=\QQ(\sqrt{5})$ and
of order $4$ for $k=\QQ(\sqrt{2})$. Case by case analysis leads
to the final statement; to check the splitting behavior in $k(\xi_{n})$
one can use the criterion of Shavel (see \cite{Shavel78}, Theorem
4.8). We double-checked the the conditions of Lemma \ref{torsions}
explicitly with PARI/GP.

\subsection{Admissible groups defined by congruences}

Let $k$ a totally real number field and $A(k,\pP_{1},\overline{\pP}_{1},\ldots,\pP_{m}\overline{\pP}_{m})$
an indefinite quaternion algebra over $k$, $\OO$ a maximal order
in $A$, $\OO$ and $\Gamma_{\OO}(1)$ as in the previous sections.
If $\aaa$ is a two-sided $\OO$ ideal in $\OO$, the \emph{principal
congruence subgroup} in $\OO(1)$ associated with $\aaa$ is defined
as 
\[
\OO(\aaa)=\{g\in\OO\mid Nrd(g)=1,g-1\in\aaa\}
\]
 Additionally we define $\Gamma_{\OO}(\aaa)=\OO(\aaa)/Z$ where $Z$
denotes the center of $\OO(\aaa)$. A \emph{congruence subgroup} in
$\mathcal{O}(1)$, resp.~$\Gamma_{\OO}(1)$, is a subgroup $G\subset\OO(1)$,
resp.~$\Gamma\subset\Gamma_{\OO}(1)$, which contains some $\OO(\aaa)$,
resp.~$\Gamma_{\OO}(\aaa)$. The group $\OO(\aaa)$ is a normal subgroup
of finite index in $\OO(1)$ and we have $\OO(1)/\OO(\aaa)\cong\Gamma_{\OO}(1)/\pm\Gamma_{\OO}(\aaa)$
if $2\notin\aaa$ and $\OO(1)/\OO(\aaa)\cong\Gamma_{\OO}(1)/\Gamma_{\OO}(\aaa)$
if $2\in\aaa$. The size $|\OO(1)/\OO(\aaa)|$ is computed as follows 
\begin{itemize}
\item Any two-sided ideal $\aaa$ in $\OO$ has a unique decomposition $\aaa=\qQQ_{1}^{e_{1}}\cdot\ldots\cdot\qQQ_{r}^{e_{r}}$
as a product of prime ideal powers. Then, $\OO(1)/\OO(\aaa)$ is a
direct product 
\[
\OO(1)/\OO(\aaa)=\OO(1)/\OO(\qQQ_{1}^{e_{1}})\times\ldots\times\OO(1)/\OO(\qQQ_{r}^{e_{r}}).
\]
 
\item Let $\qQQ$ be a prime ideal in $\OO$. The $\OO_{k}$-ideal $\qQ=Nrd(\qQQ)$
generated by the reduced norms of elements in $\qQQ$, which is also
the intersection $\qQ=\mathfrak{Q\cap\OO}_{k}$, is a prime ideal
and there are two possible cases: 

\begin{itemize}
\item $\qQ\notin Ram(A)$. Then, $\qQQ=\qQ\OO$ and $\OO(1)/\OO(\qQQ_{1}^{e})\cong SL_{2}(\OO_{k}/\qQ^{e})$. 
\item $\qQ\in Ram(A)$. Then, $\qQQ^{2}=\qQ\OO$ and $\OO(1)/\OO(\qQQ^{e})\cong(\OO/\qQQ^{e})_{1}=\ker\left((\OO/\qQQ^{e})^{\ast}\stackrel{Nr}{\rightarrow}(\OO_{k}/\qQ^{e})^{\ast}\right)$,
where $Nr$ is the norm map induced by the reduced norm $Nrd:\OO\rightarrow\OO_{k}$. 
\end{itemize}
\end{itemize}
\begin{rem}
If we want to search for admissible groups among the principal congruence
subgroups, then we must note the following: for $g\in\OO(\qQQ)$ we
have $\overline{g}\in\OO(\overline{\qQQ})$, so, $\Gamma_{\OO}(\qQQ)$
will be admissible if and only if $\qQQ=\overline{\qQQ}$ (for more
precise statement see Theorem \ref{invariantthm} and Lemma \ref{invarianceunderinvolution}
below). This is already a strong condition: If $k$ is a quadratic
field, the prime $\qQ$ under $\qQQ$ must be inert or ramified over
$\ell$. For instance, this in combination with list of possible candidates
from Theorem \ref{list_real_quadratic_candidates} shows there are
no admissible principal congruence subgroups of any type $e$ defined
over a real quadratic field. 
\end{rem}
In the following we will make use of the following well-known fact. 
\begin{lem}
\label{congruencetorsion} Let $\qQ$ be a prime ideal in $\mathcal{O}_{k}$,
unramified in a $k$-central quaternion algebra $A$ and $\mathfrak{Q=\qQ\OO}$
the corresponding $\OO$-ideal . Let $q\ZZ=\qQ\cap\ZZ$ be the rational
prime divisible by $\qQ$ and finally, let $x\in\OO^{1}(\mathfrak{Q})$
be an element of order $p$, where $p$ is a prime. Then $p=q$. \end{lem}
\begin{proof}
We have $x\in\OO^{1}(\mathfrak{Q})\Leftrightarrow Nrd(x-1)\in\qQ$.
We can assume that $x$ is a primitive $p$-th root of unity contained
in $A$. Since $k(x)\subset A$, we have $Nrd(x-1)=N_{k(x)/k}(x-1)$.
Taking $N_{k/\QQ}(\cdot)$ on both sides we obtain 
\[
N_{k/\QQ}(Nrd(x-1))=N_{k(x)/\QQ}(x-1)\in N_{k/\QQ}(\qQ)=q^{f}\ZZ.
\]
 where $f>0$ is the inertia degree of $\qQ.$ On the other hand $x-1\in\QQ(x)$,
and therefore $N_{k(x)/\QQ}(x-1)=N_{\QQ(x)/\QQ}(x-1)^{d}$, where
$d=[k(x):\QQ(x)]$. Altogether, we obtain the relation 
\[
N_{\QQ(x)/\QQ}(x-1)^{d}\in q^{f}\ZZ.
\]
 Finally, it is well-known that $N_{\QQ(x)/\QQ}(x-1)=\pm p$ and from
this the claim follows. 
\end{proof}
Assume that the prime ideal $\qQ\subset\OO_{k}$ is unramified in
$A$, then $\qQ\OO$ is a prime ideal in $\OO\subset A$. Since in
this case $\mathfrak{Q=\qQ\OO}$ we will write $\Gamma_{\OO}(\mathfrak{\qQ})=\Gamma_{\OO}(\mathfrak{Q}).$
Let $s=q^{f}$ be the absolute norm $N_{k/\QQ}(\qQ)=|\OO_{k}/\qQ|$
of $\qQ$. Then $\Gamma_{\OO}(1)/\Gamma_{\OO}(\qQ)\cong PSL_{2}(\OO_{k}/\qQ)\cong PSL_{2}(s)=PSL_{2}(\FF_{s})$.
By the classification theorem of Dickson, we know all the subgroups
of $PSL_{2}(s)$. Let us mention two particular subgroups which we
will use later on: 
\begin{enumerate}
\item Borel subgroup $B\subset PSL_{2}(s)$ consisting of all upper triangular
matrices in $PSL_{2}(s)$. The group $B$ is a maximal subgroup of
$PSL_{2}(s)$ of index $s+1$ and order $s(s-1)/t$, with $t=gcd(s-1,2)$. 
\item Unipotent subgroup $U\subset PSL_{2}(s)$ consisting of all elements
in $B$ of the form $\left(\begin{smallmatrix}1 & \ast\\
0 & 1
\end{smallmatrix}\right)$. The group $U$ is a subgroup of index $(s^{2}-1)/t$ and order $s$. 
\end{enumerate}
With above notations let $\pi:\Gamma_{\OO}(1)\longrightarrow PSL_{2}(s)$
denote the epimorphism induced by the canonical projection $\Gamma_{\OO}(1)\longrightarrow Q=\Gamma_{\OO}(1)/\Gamma_{\OO}(\qQ)$.
Let $\Gamma^{B}(\qQ)=\pi^{-1}(B)$ and $\Gamma^{U}(\qQ)=\pi^{-1}(U)$
be the inverse images of $B$ and $U$ respectively. These are subgroups
of $\Gamma_{\OO}(1)$ of index equal to the index of its image in
$PSL_{2}(s)$ under $\pi$. It is important to mention that $\Gamma_{\OO}^{B}(\qQ)$
is also constructed as the group of the norm-1 elements (modulo center)
in an appropriate Eichler order $\mathcal{E}$. The construction is
as follows: Let $\OO$ be a maximal order and denote $\OO_{v}=\OO\otimes_{\OO_{k}}R_{v}$
the localizations of $\OO$ at finite places $v$ of $k$. There $R_{v}$
denotes the valuation ring in the localization $k_{v}$ of $k$ at
$v$. Note that $\OO=\bigcap_{v}\OO_{v}$. Let $\OO'$ be another
maximal order with the property that for all finite places $v\neq\qQ$
we have $\OO_{v}=\OO'_{v}$ and additionally $\OO_{\qQ}\cap\OO'_{\qQ}$
has index $N(\qQ)=\#R_{\qQ}/\qQ R_{\qQ}$ in both, $\OO_{\qQ}$ and
$\OO'_{\qQ}$. Put $\mathcal{E=\OO\cap\OO}'$. Then, by definition,
$\mathcal{E}$ is an Eichler order of level $\qQ$. If $\qQ$ is ramified
in $A$, we have $\mathcal{E=\OO}$ and in the case where $\qQ$ is
unramified, after a possible conjugation, we can assume that $\OO_{\qQ}=M_{2}(R_{\qQ})$
and we can choose $\OO'_{\qQ}=PM_{2}(R_{\mathfrak{\mathfrak{q}}})P^{-1}$
with $P=diag(1,\varpi)$, where $\varpi$ is a generator of the valuation
ideal $\mathfrak{q}R_{\mathfrak{q}}$. The reduction modulo $\mathfrak{q}$
maps $\mathcal{E_{\qQ}=}\OO_{\qQ}\cap\OO'_{\qQ}$ surjectively to
the subalgebra of upper triangular matrices in $M_{2}(R_{\mathfrak{q}}/\mathfrak{q}R_{\mathfrak{q}})$.
Therefore the group $\mathcal{E}^{1}$ of norm-1 elements in $\mathcal{E}$
corresponds to exactly those elements in $\OO^{1}$ which modulo $\qQ\OO$
are upper triangular.

In general, a $k/\ell$-involution on a quaternion algebra $A$ does
not preserve a maximal order. But under certain conditions on $A$
we can ensure the existence of such an order: 
\begin{thm}
(Scharlau, \cite[Theorem 4.6]{Sharlau}) \label{invariantthm} Let
$A$ be a quaternion algebra over $k$ admitting a $k/\ell$-involution
$\tau$. Then there exists a maximal order $\OO$ invariant under
$\tau$ unless the following exceptional situation is given: 
\begin{itemize}
\item the extension $k/\ell$ is unramified and 
\item the number of places $v\in Ram(A)$ is $\equiv2\bmod4$. 
\end{itemize}
\end{thm}
\begin{cor}
\label{invarianceunderinvolution} Assume that $A$ admits a $k/\ell$-involution
$\tau$ which preserves a maximal order $\OO$ and let $\qQ\subset\OO_{k}$
be a prime ideal which is unramified in $A$ and $\qQ\OO$ the corresponding
prime ideal in $\OO$. Assume that the non-trivial $k/\ell$-automorphism
$x\mapsto\overline{x}$ maps $\qQ$ to itself, that is, $\overline{\qQ}=\qQ$.
Then, for each of the groups $\Gamma=\Gamma_{\OO}(1)$, $\Gamma_{\OO}^{B}(\qQ)$,
$\Gamma_{\OO}^{U}(\qQ)$ and $\Gamma_{\OO}(\qQ)$ there exists a $k/\ell$-involution
on $A$ leaving $\Gamma$ invariant.\end{cor}
\begin{proof}
The group $\Gamma_{\OO}(1)$ is $\tau$-invariant since $\OO$ is
$\tau$-invariant and for any $a\in A$ we have $Nrd(\overline{a})=\overline{Nrd(a)}$.
Also, $\Gamma(\qQ)$ is invariant, since for every $x\in\Gamma_{\OO}(\qQ)$
there is a representative $x'\in\OO$ of the class $x$ satisfying
$x-1\in\qQ\OO$, thus $\overline{(x'-1)}\in\overline{\qQ\OO}=\overline{\qQ}\OO=\qQ\OO$,
by assumption, hence $\overline{x'}\equiv1\bmod\qQ\OO$. 

In order to prove the invariance of other groups, we first localize
at $\mathfrak{q}$; since $\mathfrak{q\notin}Ram(A)$, the local algebra
$A_{\mathfrak{q}}=A\otimes_{k}k_{\mathfrak{q}}$ is isomorphic to
$M_{2}(k_{\mathfrak{q}})$ and since $\mathcal{O}$ is maximal $\mathcal{O_{\mathfrak{q}}=\OO\otimes}_{\OO_{k}}R_{\qQ}$
is isomorphic to $M_{2}(R_{\mathfrak{q}})$ where $R_{\mathfrak{q}}$
is the valuation ring in $k_{\mathfrak{q}}$. Choosing the appropriate
isomorphism $A_{\mathfrak{q}}\cong M_{2}(k_{\mathfrak{q}})$ we can
assume that $\mathcal{O_{\mathfrak{q}}}=M_{2}(R_{\mathfrak{q}})$.
Consider the order $\mathcal{\mathcal{E_{\qQ}}=}M_{2}(R_{\mathfrak{\mathfrak{q}}})\cap PM_{2}(R_{\mathfrak{\mathfrak{q}}})P^{-1}$
with $P=diag(1,\varpi)$, where $\varpi$ is a generator of the valuation
ideal $\mathfrak{q}R_{\mathfrak{q}}$. This is the localization of
the global Eichler order $\mathcal{E}$ corresponding to a group $\Gamma_{\OO}^{B}(\qQ)$.
The involution $\tau$ which leaves $\OO$ invariant extends to an
involution $\hat{\tau}$ on $\OO_{\qQ}=\OO\otimes_{\OO_{k}}R_{\qQ}$
in an obvious way by defining $\hat{\tau}(x\otimes r)=\tau(x)\otimes\bar{r}$
where $r\mapsto\overline{r}$ is the generator of $Gal(k_{\mathfrak{q}}/\ell_{\mathcal{Q}})$
and $\mathcal{Q=\qQ\cap\ell}$ is the prime of $\ell$ lying under
$\qQ$ and $\ell_{\mathcal{Q}}$ its localization. By this the involution
$\hat{\tau}$ maps the matrix $P$ to $\pm P$ depending on whether
$k_{\mathfrak{q}}/\ell_{\mathcal{Q}}$ is unramified or ramified but
in any case $\hat{\tau}$ preserves $\mathcal{E_{\qQ}}.$ From the
construction of $\hat{\tau}$ we see that $\hat{\tau}$ also preserves
$\OO\cap\mathcal{E_{\qQ}=\mathcal{E}}$ and particularly the norm-1
group $\mathcal{E}^{1}$ whose quotient by the center is $\Gamma_{\OO}^{B}(\qQ)$.
The group $\Gamma_{\OO}^{U}(\qQ)$ consists of those elements in $\Gamma_{\OO}^{B}(\qQ)$
which reduce modulo $\qQ$ to upper triangular matrices with only
1 on the diagonal. The preimage of such matrices in $\mathcal{E_{\qQ}}$
is preserved by $\hat{\tau}$ and hence $\tau$ preserves the preimage
of these matrices in $\OO\cap\mathcal{E_{\qQ}}$. This implies that
with $\Gamma_{\OO}^{B}(\text{\ensuremath{\qQ}})$ also $\Gamma_{\OO}^{U}(\qQ)$
is preserved. 
\end{proof}

\subsection{Construction with the Borel subgroup}

Let $A(k,\pP_{1},\overline{\pP}_{1},\ldots,\pP_{m},\overline{\pP}_{m})$
be as before. Let $\qQ$ be a prime ideal of $k$ such that $\qQ\neq\pP_{i},\overline{\pP}_{i}$
for $i=1,\ldots,m$ and consider the group $\Gamma_{\OO}^{B}(\qQ)$,
the inverse image $\pi^{-1}(B)$ of a Borel subgroup $B\subset\Gamma_{\OO}(1)/\Gamma_{\OO}(\qQ)\cong PSL_{2}(\OO_{k}/\qQ)$.
The group $\Gamma_{\OO}^{B}(\qQ)$ is a subgroup of index $N_{k/\QQ}(\qQ)+1$
in $\Gamma_{\OO}(1)$. In order to discuss the torsions in $\Gamma_{\OO}^{B}(\qQ)$
it will again be useful to interpret $\Gamma_{\OO}^{B}(\qQ)$ as the
norm-1 group of an Eichler order as explained in previous section.
Let us give conditions under which $\Gamma_{\OO}^{B}(\qQ)$ is torsion-free. 
\begin{lem}
\label{eichlercrit} Let $k$, $A$, $\OO$ and $\qQ$ be as above.
Then, $\Gamma_{\OO}^{B}(\qQ)$ contains a torsion if and only if a
primitive $n$-th root of unity $\xi$ can be embedded in $\mathcal{E}(\qQ)$.
This happens if and only if every prime $\pP\in Ram(A)$ is either
ramified or inert in $k(\xi)$ and $\qQ$ is split in $k(\xi)$. \end{lem}
\begin{proof}
Let $\gamma$ be a torsion in $\Gamma_{\OO}^{B}(\qQ)$. Then there
is a minimal $N$ such that $\gamma^{N}=\pm1$, which implies that
$\gamma$ is an $N$-th (or $2N$-nth) root of unity contained in
$\mathcal{E}(\qQ)$. Conversely, let $\xi$ be a root of unity, let
$L=k(\xi)$ and assume that there exists an embedding $\sigma:L\hookrightarrow A$
such that $\sigma(L)\cap\mathcal{E}(\qQ)=\mathcal{O}_{k}(\xi)$ (that
is, $\mathcal{O}_{k}(\xi)$ is embedded in $\mathcal{E}(\qQ)$). Then
$\sigma(\xi)$ is a torsion in $\Gamma_{\OO}^{B}(\qQ)$. By a theorem
of Eichler (see \cite[Satz 6]{eichler}) such an embedding is possible
if and only if the splitting condition mentioned in the statement
of Lemma is satisfied. 
\end{proof}

\subsection{A Shimura surface with an involution of second kind and $p_{g}=5$.\label{sub:pg5}}

Let $k=\QQ(\sqrt{33})$ and $A=A(k,\pP_{2}\overline{\pP}_{2})$ the
indefinite quaternion algebra ramified exactly at the two places over
$2$ (note that $2$ is split in $k$, since $33\equiv1\bmod8$).
By Theorem \ref{existence_type_2}, $A$ admits a $k/\QQ$-involution
and since $k/\QQ$ is not totally ramified, Theorem \ref{invariantthm}
ensures the existence of an involution invariant order $\OO$. Let
$\qQ=\qQ_{11}$ be the prime over $11$. Since $11$ is ramified in
$k$, we have $N_{k/\QQ}(\qQ)=11$ and $\Gamma_{\OO}^{B}(\qQ_{11})$
is of index $12$ in $\Gamma_{\OO}$. By the volume formula from Theorem
\ref{shimizu}, we have $e(\Gamma_{\OO}(1))=2$, hence $e(\Gamma_{\OO}^{B}(\qQ))=24$.
Let us show that $\Gamma_{\OO}^{B}(\qQ)$ is torsion-free. For this
we need to exclude the existence of elements of order $2$ and $3$
only, since these are the only primes for which an embedding of $\xi_{p}$
in $A$ is possible. Elements of order $2$ come from embeddings of
$k(\xi_{4})=k(\sqrt{-1})$ in $A$ and those of order $3$ from embeddings
of $k(\xi_{6})=k(\sqrt{-3})$. We use Lemma \ref{eichlercrit}: $k(\xi_{4})\cong\QQ[x]/\langle x^{4}-64x^{2}+1156\rangle$
and we find that $11\OO_{k(\zeta_{4})}=\mathfrak{Q}^{2}$ with $(\OO_{k(\xi_{4})}/\mathfrak{Q}:\FF_{11})=2$.
It follows that $\qQ_{11}$ is inert in $k(\xi_{4})$ and by Lemma
\ref{eichlercrit}, $\Gamma_{\OO}^{B}(\qQ)$ contains no elements
of order $2$. Similar argument excludes the existence of elements
of order $3$. Namely, $k(\xi_{6})\cong\QQ[x]/\langle x^{4}-60x^{2}+1296\rangle$
and in $k(\xi_{6})$ we again have $11\OO_{k(\zeta_{6})}=\mathfrak{Q}^{2}$
with a prime ideal $\mathfrak{Q}$ in $\OO_{k(\zeta_{6})}$ whose
inertia degree is $(\OO_{k(\xi_{6})}/\mathfrak{Q}:\FF_{11})=2$. Again
this implies that $\qQ_{11}$ is inert in $\OO_{k(\xi_{6})}$ and
by Lemma \ref{eichlercrit}, there are no elements of order 3 in $\Gamma_{\mm}^{B}(\qQ_{11})$.
Finally by Corollary \ref{invarianceunderinvolution}, $\Gamma_{\OO}^{B}(\qQ)$
is invariant under the involution on $\OO$ and we get: 
\begin{thm}
The group $\Gamma_{\OO}^{B}(\qQ_{11})$ is admissible of type $24$. \end{thm}
\begin{rem}
\label{remark_deg_2}Unfortunately, one promising candidate for a
admissible group of type $12$ ($p_{g}=2$) fails to be torsion-free.
Namely, let $A=A(\QQ(\sqrt{17}),\pP_{2}\overline{\pP}_{2})$ and take
$\qQ=\qQ_{17}$, the ideal over $17$. Then the index of $\Gamma_{\OO}^{B}(\pP_{17})$
in $\Gamma_{\OO}(1)$ is $18$ and as $e(\Gamma_{\OO}(1))=2/3$, we
get $e(\Gamma_{\OO}^{B}(\pP_{17}))=12$. The invariance under the
involution of second kind is guaranteed by the condition $\qQ_{17}=\overline{\qQ}_{17}$.
But in $L=k(\xi_{4})\cong\QQ[x]/\langle x^{4}-32x^{2}+324\rangle$,
both primes $\pP_{2}$ and $\overline{\pP}_{2}$ are non-split and
$\pP_{17}$ is split. This implies that there are 2-torsions in $\Gamma_{\OO}^{B}(\qQ_{17})$.
There are more examples of non-smooth Shimura surfaces with ``good''
invariants. Consider for instance $k=\QQ(\sqrt{28})$ and $A=A(k,\pP_{3}\overline{\pP}_{3})$
the indefinite quaternion algebra ramified at the two places over
$3$. Theorems \ref{existence_type_2} and \ref{invarianceunderinvolution}
ensure that $A$ has an involution of second kind and that there is
an order $\OO$ invariant under the involution. From Theorem \ref{shimizu}
we know that $e(\Gamma_{\OO}(1))=16/3$. The rational prime $2$ is
ramified in $k$. Let $\qQ=\qQ_{2}$ be the prime ideal of $k$ with
$\qQ_{2}^{2}=2\OO_{k}$ and consider the principal congruence subgroup
$\Gamma_{\OO}(\qQ)$. It is a subgroup in $\Gamma_{\text{\ensuremath{\OO}}}(1)$
of index $6$. There are no elements of order $3$ in $\Gamma_{\OO}(\qQ)$
by Lemma \ref{congruencetorsion} but there are elements of order
2 coming from embeddings of $\sqrt{-1}$ in $\OO$. \\

Looking at the list in Theorem \ref{list_real_quadratic_candidates},
we can also prove that no other admissible groups of type $e=12,\ldots,36$
can be obtained from $k=\QQ(\sqrt{d})$ and $\Gamma=\Gamma_{\OO}^{B}(\qQ)$
or $\Gamma_{\OO}^{U}(\qQ)$. 
\end{rem}
Instead we can consider totally real fields of higher degree:

\subsection{A Shimura surface with an involution of second kind and $p_{g}=6$.\label{sub:pg6}}

In this example we consider the unique totally real number field $k$
of degree $4$ and discriminant $d_{k}=725$. Its defining polynomial
is $x^{4}-x^{3}-3x^{2}+x+1$ and $k$ contains $\ell=\QQ(\sqrt{5})$
as a subfield of degree $2$. Let us consider the $k$-central quaternion
algebra $A(k,\emptyset)$ ramified exactly at two infinite places
$v_{1}$ and $v_{2}$ of $k$ such that $v_{2}=v_{1}\circ\sigma$,
where $\langle\sigma\rangle=Gal(k/\ell)$. We remark that $k$ is
not a Galois extension of $\QQ$. The algebra $A$ admits an involution
of second kind $\tau$ and by Theorem \ref{invariantthm}, there is
a maximal order $\OO$ invariant under $\tau$. Consider now the prime
$q=29$. In $\ell=\QQ(\sqrt{5})$, $29\OO_{\ell}=\mathcal{Q}_{29}\mathcal{Q}_{29}'$
is a product of two primes. On the other hand, a computation with
PARI/GP shows that the ideal $29\OO_{k}=\qQ_{29}^{2}\qQ_{29}'$ is
also a product of two prime ideals $\qQ_{29}$ (with multiplicity
2) and $\qQ_{29}'$, hence neither $\mathcal{Q}_{29}$ nor $\mathcal{Q}_{29}'$
is split in $k$. Moreover we deduce that $\qQ_{29}^{2}=\mathcal{Q}_{29}\OO_{k}$
and $\qQ_{29}'=\mathcal{Q}'_{29}\OO_{k}$ as well as $\OO_{k}/\qQ_{29}\cong\mathbb{F}_{29}$
and $\OO_{k}/\qQ'_{29}\cong\mathbb{F}_{29^{2}}$. By Theorem \ref{shimizu}
we have $e(X_{\Gamma_{\OO}(1)})=1/15$ (we compute $\zeta_{k}(2)$
with PARI/GP command \verb  zetak). Consider the congruence subgroup
$\Gamma_{\OO}^{U}(\qQ_{29})$. We obtain $[\Gamma_{\OO}(1):\Gamma_{\OO}^{U}(\qQ_{29})]=420$
and $c_{2}(X_{\Gamma_{\OO}^{U}(\qQ_{29})})=28$. By corollary \ref{invarianceunderinvolution},
$\Gamma_{\OO}^{U}(\qQ_{29})$ is stable under $\tau$. Also $\Gamma_{\OO}^{U}(\qQ_{29})$
is torsion-free. Namely, as the order of $U$ is $s$, any non-trivial
torsion element in $\Gamma_{\OO}^{U}(\qQ_{29})$ has order $29$ (which
is impossible by lemma \ref{torsions}) or lies already in $\Gamma_{\OO}(\qQ_{29})$.
But this latter group is torsion-free by Lemma \ref{congruencetorsion}.
\begin{rem}
\label{remark_in_degree_4} Of course, the strategy of Section \ref{sub:Smoothness-and-the-Euler}
leading to Theorem \ref{list_real_quadratic_candidates} could be
applied also in the case of quaternion algebras over totally real
fields $k$ of degree $>2$ but becomes very soon computationally
involved. Restricting ourselves to totally real quartic fields of
discriminant $\leq10^{4}$ and groups of type $\Gamma_{\OO}^{B}(\qQ)$
or $\Gamma_{\OO}^{U}(\qQ)$ we find that example \ref{sub:pg6} is
the only example of an admissible group (of any type $e=12+4k$, $0\leq k\leq6$).
But similarly to Remark \ref{remark_deg_2} we find some interesting
non-smooth examples: Let $k_{4,D}$ denote a totally real field of
degree $4$ and discriminant $D$ (in the examples below, there will
be only one such field up to isomorphism) containing a real quadratic
field $\ell$. Let $A(k_{4,D},\emptyset)$ denote the quaternion algebra
over $k_{4,D}$ ramified exactly at the two infinite places of $k_{4,D}$
which are conjugate under the non-trivial $\ell$-automorphism of
$k_{4,D}$. Such quaternion algebra admits a $k_{4,D}/\ell$ -involution.
Let $\OO$ be a maximal order in $A(k_{4,D},\emptyset)$ and $\Gamma_{\OO}(1)$
the corresponding projectivized modular group. We get several \emph{singular}
Shimura surfaces $X_{\Gamma}$ admitting an involution of second kind
with $\Gamma\subset\Gamma_{\OO}(1)$ given in the table below:
\end{rem}
\begin{tabular}{|c|c|c|c|c|}
\hline 
$k_{n,D}$ & $\ell$ & $e(\Gamma(1))$ & $\Gamma$ & $e(\Gamma)$\tabularnewline
\hline 
\hline 
$k_{4,2624}$ & $\mathbb{Q}(\sqrt{5})$ & $1/2$ & $\Gamma^{B}(\qQ_{7,1})\cap\Gamma^{B}(\text{\ensuremath{\qQ}}_{7,2})$ & 32\tabularnewline
\hline 
$k_{4,2000}$ & $\mathbb{Q}(\sqrt{5})$ & $1/3$ & $\Gamma(\qQ_{5})$ & 20\tabularnewline
\hline 
$k_{4,2000}$ & $\mathbb{Q}(\sqrt{5})$ & $1/3$ & $\Gamma(\qQ_{2})$ & 20\tabularnewline
\hline 
$k_{4,2525}$ & $\mathbb{Q}(\sqrt{5})$ & $2/3$ & $\Gamma^{B}(\qQ_{5})\cap\Gamma^{B}(\overline{\qQ}_{5})$ & 24\tabularnewline
\hline 
$k_{4,3600}$ & $\mathbb{Q}(\sqrt{3})$,  & $4/5$ & $\Gamma^{U}(\qQ_{2})$ & 12\tabularnewline
\hline 
\end{tabular}

\section{determination of the fixed curve.\label{sec:determination-of-the-curve}}

Let $X_{\Gamma}=\HH_{\CC}^{2}/\Gamma$ be a smooth Shimura surface
such that the involution $\mu$ on $\HH_{\CC}^{2}$ exchanging the
factors descends to an involution $\sigma$ on the quotient $X_{\Gamma}$.
The image of the diagonal $\Delta\subset\HH\times\HH$ is a smooth
Shimura curve $C_{\Gamma}$ fixed by $\sigma$. The aim of this section
is to determine that curve in examples we investigated in Sections
\ref{sub:pg5} and \ref{sub:pg6}. 

The analogous problem for Hilbert modular surfaces is well-known,
see for instance \cite{HirzebruchModular} or \cite{Hausmann}. The
quaternion algebra $A$ over $k$ defining the group $\Gamma$ has
a non trivial involution $\tau$ of second kind. That involution leaves
invariant a subfield $\ell$. Recall that $A=A(k\mathfrak{,p}_{1},\dots,\mathfrak{p}_{2r})$
with $\mathfrak{p}_{2i-1}=\mathfrak{p}_{2i}^{\alpha}$ as in Prop.
\ref{existence_type_2} and $\alpha:k\to k$ the involution of the
extension $k/\ell$ with $\sigma_{1}=\sigma_{2}\circ\alpha$ for $\sigma_{i}:k\to\RR$
the unramified infinite places in $A$. Also note that the involution
$\sigma:X_{\Gamma}\longrightarrow X_{\Gamma}$ is determined by the
involution of second kind $\tau$ on $A$. The fixed point set of
$\sigma$ is associated with the invariant $\ell$-subalgebra $A'=\left\{ a\in A\ \mid\ \tau(a)=a\right\} $
of $A$. From the proof of Proposition \ref{existence_type_2} we
know that $A'$ is a quaternion algebra over $\ell$ with the property
$A'\otimes k\cong A$. Moreover, $A'$ is ramified at every prime
$\mathcal{P}_{i}$ of $\ell$ such that $\mathcal{P}_{i}\OO_{k}=\pP_{2i-1}\pP_{2i}$
for $i=1,\ldots,r$.
\begin{lem}
Let $A=A(k,\pP_{1},\ldots,\pP_{2r})$ be a quaternion algebra admitting
an involution of second kind $\tau$ and $A'$ the elementwise $\tau$-invariant
subalgebra. Let $\OO$ be a order in $A$, then $\OO'=\OO\cap A'$
is an order in $A'$. Conversely, assume that $A'=A'(\ell,\mathcal{P}_{1},\ldots,\mathcal{P}_{s})$
is a quaternion algebra over $\ell$ and $\OO'$ an order in $A'$
then $\OO=\OO'\otimes_{\OO_{\ell}}\OO_{k}$ is an order in $A.$ Assume
that $\OO'$ is maximal and each $\mathcal{P}_{i}$ is split in $k$
then $\OO'\otimes_{\OO_{\ell}}\OO_{k}$ is a maximal order in $A=A'\otimes_{\ell}k$.
$ $\end{lem}
\begin{proof}
The first part of the Lemma concerning the correspondence between
orders $\OO'$ in $A'$ and orders in $A$ is obvious and we shall
therefore prove only the second part. Assume that $\OO'\subset A'$
is a maximal order and let $\OO=\OO'\otimes_{\OO_{\ell}}\OO_{k}$.
The order $\OO$ is maximal if and only if each of its localizations
$\OO_{\pP}$ is maximal in the local algebra $A_{\pP}$. Here, $\OO_{\pP}=\OO'_{\mathcal{P}}\otimes\OO_{k_{\pP}}$
arises from the local maximal order $\OO'_{\mathcal{P}}$ corresponding
to a finite place $\mathcal{P}$ of $\ell$ lying under $\pP$. Assume
now that $\mathcal{P\neq\mathcal{P}}_{i}$ is a finite place at which
$B$ is unramified. Then $\OO'_{\mathcal{P}}\cong M_{2}(\OO_{\ell_{\mathcal{P}}})$
and clearly $\OO_{\pP}\cong M_{2}(\OO_{k_{\pP}})$ is also maximal.
If $\mathcal{P=\mathcal{P}}_{i}$ is a place such that $A'_{\mathcal{P}_{i}}$
is a division algebra, as by assumption $k_{\pP_{i}}=k_{\bar{\pP_{i}}}=\ell_{\mathcal{P}_{i}}$,
$\OO_{\pP_{i}}$ and $\OO_{\overline{\pP}_{i}}$ are maximal.\end{proof}
\begin{rem}
We shall note that in the case where $A'=A'(\ell,\mathcal{P}_{1},\ldots,\mathcal{P}_{1},\mathcal{Q})$
ramifies also at some prime $\mathcal{Q}$ that is non-split in $k$
the order $\OO=\OO'\otimes\OO_{k}$ is not maximal even if $\OO'$
is maximal. \end{rem}
\begin{example}
Let $A'=\left(\frac{2,5}{\QQ}\right)$ be the quaternion algebra over
$\QQ$ generated by elements $1,i,j,ij$ such that $i^{2}=2$, $j^{2}=5$
(and $ij=-ji$). The algebra $A'$ is ramified exactly at the primes
$2$ and $5$. Let $\OO'$ be a maximal order in $A'$. The group
$\Gamma_{\OO'}(1)$ is a Fuchsian group. Let $\Gamma_{\OO'}^{B}(11)$
be the subgroup corresponding to the Borel subgroup of $PSL_{2}(\FF_{11})$.
This subgroup can be interpreted as the group of elements of reduced
norm 1 of an Eichler order $\mathcal{E}(11)$ of level $11$. The
group $\Gamma_{\OO'}(1)$ is of index $12$ in $\Gamma_{\OO'}(1)$
and is torsion-free by Lemma \ref{eichlercrit}, as $5$ is split
in $\QQ(i)$ and $11$ is non-split in $\QQ(\sqrt{3})$. The genus
of the curve $C=\Gamma_{\OO'}(11)\backslash\mathbb{H}$ can be easily
computed with the already used general volume formula from \cite{shimizu}
(see also \cite[III, Prop. 2.10]{vign2}) by which $2-2g(C)=-8$.
Let $k=\QQ(\sqrt{33})$. Then $A=A'\otimes k=\left(\frac{2,5}{k}\right)$
is a quaternion algebra over $k$ and is ramified exactly at the two
primes lying over $2$. The Eichler order $\mathcal{E}(11)$ in $A'$
is naturally contained in the order $\mathcal{E}(11)\otimes_{\ZZ}\OO_{k}$
of $A$ and the latter one is contained in $\mathcal{E}(\qQ_{11})$
the Eichler order of $A$ corresponding to the prime $\qQ_{11}$ of
$k$ lying over $11$, since the elements $\mathcal{E}(11)\otimes_{\ZZ}\OO_{k}$
become upper triangular modulo $11$, hence modulo $\qQ_{11}$. This
gives an embedding of $\Gamma_{\OO'}^{B}(11)$ in $\Gamma_{\OO}^{B}(\qQ_{11})$
and hence an embedding of a Shimura curve of genus $5$ into the Shimura
surface $X=\Gamma_{\OO}^{B}(\qQ_{11})\backslash\HH\times\HH$ (see
Section \ref{sub:pg5}) which by construction must be fixed by the
involution of second kind on $X$. This gives a precise characterization
of the Shimura surface $Z=X/\sigma$, the quotient of $X$ by the
involution on $X$ induced by the involution of second kind:\end{example}
\begin{prop}
The surface $Z$ is a smooth surface of general type with $p_{g}=0$,
$K_{Z}^{2}=4$ and $e(Z)=8$.
\end{prop}

\bibliographystyle{amsalpha}
\bibliography{refs}

\vspace{0.2in}
 {\large{\setlength{\parindent}{0.2in} }}{\large \par}

{\large{Amir D\v{z}ambi\'{c},}}{\large \par}

{\large{Johann Wolfgang Goethe Universität, Institut für Mathematik,}}{\large \par}

{\large{Robert-Mayer-Str. 6-8, 60325 Frankfurt am Main,}}{\large \par}

{\large{Germany}}{\large \par}

\texttt{\large{dzambic@math.uni-frankfurt.de}}{\large \par}

{\large{\vspace{0.2in}
 \setlength{\parindent}{0.5in}}}{\large \par}

{\large{Xavier Roulleau,}}{\large \par}

{\large{Laboratoire de Mathématiques et Applications, Université de
Poitiers,}}{\large \par}

{\large{Téléport 2 - BP 30179 - 86962 Futuroscope Chasseneuil }}{\large \par}

{\large{France}}{\large \par}

\texttt{\large{roulleau@math.univ-poitiers.fr{} }}{\large{{} }} 
\end{document}